\documentclass[11pt,bezier]{article}
\usepackage{amsmath,amssymb,amsfonts,euscript,graphicx}

\newtheorem{preproof}{{\bf \indent Proof.}}

\newenvironment{proof}[1]{\begin{preproof}{\rm
               #1}\hfill{$\Box$}}{\end{preproof}}

\newtheorem{prop}{\bf\indent Proposition}[section]
\newtheorem{defn}[prop]{\bf\indent Definition}
\newtheorem{cor}[prop]{\bf\indent Corollary}
\newtheorem{example}[prop]{\bf\indent Example}
\newtheorem{thm}[prop]{\bf\indent Theorem}

\title{\bf  \large Cyclic-Uniform Uniserial Modules and Rings}
\author{{\normalsize  {\sc R. Nikandish${}^{\mathsf{a}}$\thanks{Corresponding author}, {\sc M.J. Nikmehr${}^{\mathsf{b}}$} and {\sc A. Yassine${}^{\mathsf{b}}$}  }
}\vspace{3mm}\\
{\footnotesize{}}\\
{\footnotesize{}}\\
{\footnotesize{${}^{\mathsf{a}}$\it Department of Mathematics, Jundi-Shapur University of Technology,}}\\
{\footnotesize{\rm P.O. BOX \rm{64615-334},
Dezful, Iran}}\\
{\footnotesize{ $\mathsf{r.nikandish@ipm.ir}$}}\\
{\footnotesize{${}^{\mathsf{b}}$\it Faculty of Mathematics, K.N. Toosi
University of Technology, }}\\
{\footnotesize{\rm P.O. BOX \rm{16315-1618}, Tehran, Iran}}\\
{\footnotesize{ $\mathsf{nikmehr@kntu.ac.ir}$}}\quad\quad
{\footnotesize{$\mathsf{yassine\_ali@email.kntu.ac.ir}$}}\\
{\footnotesize{$\mathsf{}$ }}}
\date{}

\begin{document}

\maketitle
\begin{abstract}
{
An $R$-module $M$ is called virtually uniserial if for every finitely generated submodule $0 \neq K \subseteq M$, $K/$Rad$(K)$ is virtually simple. In this paper, we generalize virtually uniserial modules by dropping the virtually simple condition and replacing it by the cyclic uniform condition.  An $R$-module $M$ is called cyclic-uniform uniserial if  $K/$Rad$(K)$ is cyclic and uniform, for every finitely generated submodule $0 \neq K \subseteq M$. Also, $M$ is said to be cyclic-uniform serial if it is a direct sum of cyclic-uniform uniserial modules. Several properties of cyclic-uniform (uni)serial modules and rings are given.  Moreover, the structure of Noetherian left cyclic-uniform uniserial rings are characterized. Finally, we study rings $R$ have the property that every finitely generated $R$-module is cyclic-uniform serial.

}
\end{abstract}
\begin{center}{\section{Introduction
}}\end{center}
\par
We begin with some notions and definitions from ring theory and module theory. We assume throughout this paper that all rings are associative rings with identity and all modules are unitary left modules. Let $R$ be a ring. The set of all maximal left (resp., two sided) ideals of $R$,  the set of integers and integers modulo $n$ are denoted by Max$_l(R)$ (resp., Max$(R)$),  $\mathbb{Z}$ and $\mathbb{Z}_n$, respectively. A ring $R$ is called \textit{local} if $R$ has a unique maximal left ideal. Let $M$ be an $R$-module. The intersection of all maximal submodules of $M$, denoted by Rad$(M)$,  is called the \textit{radical of} $M$. In particular, the radical of the $R$-module $_RR$, denoted by J$(R)$, is the Jacobson radical of the ring $R$, i.e., J$(R) =$ Rad$(_RR)$. A ring $R$ is called \textit{semilocal} if $R$/J($R$) is a semisimple Artinian ring.  Note that if a ring $R$ has a finite number of maximal left ideals, then $R$ is semilocal. Recall from \cite{Wisbauer} that an $R$-module $M$ is said to be \textit{uniserial} if its submodules are linearly ordered by inclusion. Also, $M$ is said to be \textit{serial} if it is a direct sum of uniserial modules. A \textit{left (resp., right) uniserial ring} is a ring which is uniserial as a left (resp., right) module. A \textit{left (resp., right) serial ring} is a ring which is serial as a left (resp., right) module. Also, a ring $R$ is called \textit{uniserial (resp., serial)} if it is both a left and a right uniserial (resp., serial) ring. A ring $R$ is called right \textit{(FGC) K\"othe} if every (finitely generated) right $R$-module is a direct sum of cyclic $R$-modules. A ring $R$ is an \textit{Artinian (resp., Noetherian)} ring if it is both left and right Artinian (resp., Noetherian). A ring $R$ is a \textit{principal ideal ring} if it is both left and right principal ideal ring. A submodule $N$ of a module $M$ is said to be \textit{essential}, if it has non-zero intersection with every non-zero submodule of $M$. A nonzero $R$-module $M$ is called \textit{uniform} if every nonzero submodule of $M$ is essential in $M$. The \textit{uniform dimension} of a module $M$, denoted u.dim($M$), is defined to be $n$ if there is an essential submodule $N\subseteq M$ such that $N$ is a direct sum of $n$ uniform submodules. If no
such essential submodule exists, then u.dim($M)  = \infty$. An $R$-module $M$ is called \textit{isonoetherian} if, for every ascending chain $M_1 \leq M_2 \leq \cdots$ of submodules of $M$, there exists an index $n \geq 1$ such that $M_n\cong M_i$ for every $i \geq n$. An $R$-module $M$ is called \textit{B\'ezout} if every finitely generated submodule of $M$ is cyclic. A left $R$-module $M$ which has a composition series is called a \textit{module of finite length}. The length of a composition series of $M$ is called the length of $M$ and denoted by \textit{length$(M)$}. A left Artinian ring $R$ is called \textit{of bounded representation type} if the length of each finitely generated indecomposable left $R$-module is bounded. A ring $R$ is said to be \textit{left pure-semisimple} if every left $R$-module is a direct sum of finitely generated $R$-modules.  A ring $R$ is called a \textit{torch ring} if $R$ is not local, has a unique minimal prime ideal $P$ and $P$ is a non-zero uniserial $R$-module, $R/P$ is an h-local domain and $R$ is a locally almost maximal B\'ezout ring. An $R$-module $M$ is said to be \textit{virtually simple} if $M \neq 0$ and $N\cong M$ for every nonzero submodule $N$ of $M$. For any undefined notation or terminology and all the basic results on rings and modules that are used in the sequel, we refer the reader to \cite{AndersonFuller,Faith,GoodearlWarfield,Lam,McConnell} and \cite{Wisbauer}. We assume throughout this paper that all modules are unitary left modules.

Determining the structure of rings for which all modules are direct sums of certain modules has a long history. For instance, a famous result due to K\"othe \cite{Kothe} states that ``over an Artinian principal ideal ring, every module is a direct sum of cyclic submodules". Cohen and Kaplansky \cite{CohenKaplansky} showed that the converse is true for a commutative ring $R$, i.e., ``for a commutative ring $R$, $R$ is an Artinian principal ideal ring if and only if every module is a direct sum of cyclic submodules". Later, Nakayama \cite{Nakayama} showed that ``for an Artinian serial ring $R$, every $R$-module is a direct sum of cyclic submodules, but the converse need not be true". After that, Asano showed that the converse is true if $R$ is a commutative ring, i.e., ``for a commutative ring $R$, $R$ is an Artinian serial ring $R$ if and only if every $R$-module is a direct sum of cyclic submodules". Another one of the important contributions in this direction is due to Nakayama and Skornyakov \cite{Nakayama,Skornyakov} states that ``$R$ is an Artinian serial ring if and only if every left (right) $R$-module is a serial module". Warfield's Theorem on Noetherian serial rings \cite{WarfieldJr} is a landmark in the theory of rings. We recall this theorem as follows: ``A Noetherian ring $R$ is serial if and only if every finitely generated left $R$-module is serial". Warfield also showed that if $R$ is a serial ring, then every finitely presented left $R$-module is a direct sum of finitely many uniserial left $R$-modules, and then, Facchin \cite{FacchiniKrull} introduced the uniqueness of such a direct-sum decomposition.

Most recently, the notions of ``virtually uniserial modules" and ``virtually serial modules" have been introduced and studied by Behboodi et al. in \cite{BehboodiVU} as generalizations of uniserial modules. We recall from \cite[Definitions 1.4]{BehboodiVU} that an $R$-module $M$ is \textit{virtually uniserial} if, for every finitely generated submodule $0 \neq K \subseteq M$, $K/$Rad$(K)$ is virtually simple. Also, an $R$-module $M$ is called \textit{virtually serial} if it is a direct sum of virtually uniserial modules. A ring $R$ is called left virtually uniserial (resp., left virtually serial) if $_RR$ is a virtually uniserial (resp., virtually serial) module. Also, a ring $R$ is called virtually uniserial (resp., virtually serial) if it is both a left and a right virtually uniserial (resp., virtually serial) ring. It follows from \cite[Proposition 55.1]{Wisbauer} that an $R$-module $M$ is uniserial if and only if for every finitely generated submodule $0 \neq K \subseteq M$, $K/$Rad$(K)$ is simple, and by \cite[Remark 2.2.]{Facchini}, every virtually simple module is a cyclic Noetherian uniform module. This motivates us to generalize (virtually uniserial) uniserial modules and rings via the notion of cyclic uniform modules.

In this paper, we continue the study of Behboodi et al. in \cite{BehboodiVU} by dropping the virtually simple condition and replacing it by the cyclic uniform condition and a new class of modules called cyclic-uniform uniserial modules (cu-uniserial modules) is  introduced (see Definition \ref{*}). So a natural question is posed: ``Which rings $R$ have the property that every finitely generated $R$-module is cu-serial?"  This paper is in this direction and it is devoted to answer this question. In section 2, some basic properties of cu-serial (cu-uniserial) modules and rings are given. It is shown that every cu-uniserial module is uniform and B\'ezout (Theorem \ref{222}). For nontrivial cu-uniserial modules see Example \ref{exx}. We identify several classes of rings over which the classes of uniserial modules (resp., ring), virtually uniserial modules (resp., ring) and cu-uniserial modules (resp., ring) are coincide (see Theorems \ref{prev} and \ref{3437}). It is proved that a cu-uniserial ring cannot be semilocal (see Corollary \ref{2.11}). We characterize the structure of Noetherian left cu-uniserial rings in Theorem \ref{vieww}. In Section 3, it shown that a ring $R$ is an Artinian serial ring if and only if every left $R$-module is cu-serial if and only if every left $R$-module is serial if and only if every left $R$-module is virtually
serial (see Corollary \ref{322}). Also, it is proved that for a ring $R$ if every left $R$-module is left cu-serial, then $R/{\rm J}(R)\cong M_{n_1}(D_1) \times \cdots \times M_{n_k}(D_k)$ for some positive integers $n_i$ and some principal left ideal domains $D_i$ (see Theorem \ref{2.113}). Also, it is shown that, for a commutative ring $R$, every finitely generated $R$-module is cu-serial if and only if $R = \prod_{i=1}^nR_i$ where $n > 0$ and the $R_i$'s are almost maximal uniserial rings, principal ideal domains with {\rm J}$(R_i) = 0$ and torch rings such that $R_i/P_i$ is a principal ideal domain with $|{\rm Max}(R_i/P_i)| = \infty$ where $P_i$ is the unique minimal prime ideal of $R_i$ (see Theorem \ref{377}). Also, it is shown that a ring $R$ is a left V-ring such that every finitely generated left $R$-module is cu-serial and every cyclic $R$-module is either virtually simple direct summand of $_RR$ or simple if and only if $R\cong M_{n_1}(D_1) \times \cdots \times M_{n_k}(D_k)$ for some positive integers $n_i$ and some principal ideal V-domains $D_i$ if and only if $R$  is a fully virtually semisimple ring if and only if every finitely generated left $R$-module is a direct sum of virtually simple $R$-modules if and only if every finitely generated left $R$-module is virtually semisimple if and only if every finitely generated left $R$-module is completely virtually semisimple (see Theorem \ref{388}).  Finally, it is proved that if $R$ is a duo ring, then $R$ is a principal ideal ring if and only if every finitely generated left $R$-module $M$ has a decomposition $M = Z\oplus P$ where $Z$ is Noetherian cu-serial and $P$ is a direct sum of projective virtually simple modules (see Theorem \ref{31111}).

\vspace*{1cm}
\begin{center}{\section{Some properties of cyclic-uniform uniserial modules and rings
}}\end{center}

In this section, we introduce cyclic-uniform uniserial (serial) modules and rings, as a nontrivial generalization of uniserial modules and rings. Moreover, some useful properties of them are given.

We start with the following definition.

\begin{defn} \label{*}
Let $R$ be a ring and $M$ an $R$-module. Then $M$ is called cyclic-uniform uniserial (cu-uniserial) if for every finitely generated submodule $0 \neq K \subseteq M$, $K/$Rad$(K)$ is cyclic and uniform. Also, $M$ is said to be cyclic-uniform serial (cu-serial) if it is a direct sum of cyclic-uniform uniserial modules. The ring $R$ is left cyclic-uniform uniserial (resp., left cyclic-uniform serial) if it is cyclic-uniform uniserial (resp., serial) as a left $R$-module. Also, $R$ is said to be cyclic-uniform uniserial (resp., cyclic-uniform serial) if it is both a left and a right cyclic-uniform uniserial (resp., cyclic-uniform serial) ring.
\end{defn}

Recall from \cite{BehboodiAlmost} that a left $R$-module $M$ is said to be \textit{almost uniserial} if for every pair $N,K$ of submodules of $M$ either $N \subseteq K$, or $K \subseteq N$, or $N \cong K$. It is easy to see that every uniserial module is cu-uniserial, but the converse need not be true in general. So, one can easily see that the $\mathbb{Z}$-module $\mathbb{Q}$ is a cu-uniserial $\mathbb{Z}$-module, but it is not a uniserial $\mathbb{Z}$-module. Also, the class of cu-uniserial modules and almost uniserial modules are not coincide, since the $\mathbb{Z}$-module $\mathbb{Q}$ is a cu-uniserial module, but is not an almost uniserial $\mathbb{Z}$-module. However, the converse is not true also, since, for example, the $\mathbb{Z}$-module $\mathbb{Z}_2\oplus \mathbb{Z}_2$ is an almost uniserial $\mathbb{Z}$-module, but is not a cu-uniserial $\mathbb{Z}$-module, because $((0\oplus \mathbb{Z}_2)/$Rad$(0\oplus \mathbb{Z}_2)) \cap ((\mathbb{Z}_2\oplus 0)/$Rad$(\mathbb{Z}_2\oplus 0)) = 0$, and so $(\mathbb{Z}_2\oplus \mathbb{Z}_2)/$Rad$(\mathbb{Z}_2\oplus \mathbb{Z}_2)$ is not uniform. In the other hand, it follows from \cite[Proposition 2.4]{BehboodiVU} taht every virtually uniserial module is cu-uniserial, but the converse need not be true in general. Here is an example of a uniform cyclic module that is not virtually simple (see \cite[Remark 6.6. (1)]{Facchini}). Let $D$ be an integral domain with $D_D$ uniform but not isonoetherian (e.g. $k[x_1, x_2, \cdots]$, where $k$ is a field). Then $D_D$ is a cylic uniform module, but it is not virtually simple, by \cite[Remark 6.6(1)]{Facchini}. Therefore, a cu-uniserial module need not be virtually uniserial. Consequently, in view of the last display, we have
\begin{center}
{
\{uniserial modules\} $\subsetneq$ \{virtually uniserial modules\} $\subsetneq$ \{cu-uniserial modules\} $\nsubseteq$ \{almost uniserial modules\}

\{uniserial modules\} $\subsetneq$ \{almost uniserial modules\} $\nsubseteq$ \{cu-uniserial modules\}
}
\end{center}

 For nontrivial cu-uniserial modules see the following examples.

\begin{example}\label{exx} {\rm

$(1)$	 The $\mathbb{Z}$-module $M = \mathbb{Z}_{p^n}$ is a finitely generated cu-uniserial left $\mathbb{Z}$-module, for every prime number $p$ and integer $n > 1$ (see Theorem \ref{2.115}).

$(2)$ The $\mathbb{Z}$-module $\mathbb{Z}$ is cu-uniserial, but it is not a uniserial module (see Theorem \ref{2.115}).

$(3)$ Every discrete valuation domain $R$ as an $R$-module is a finitely generated cu-uniserial $R$-module (see Theorem \ref{2.115}).

$(4)$  Let $I$ be an infinite proper subset of the set of all prime numbers, $S = \mathbb{Z} \setminus \bigcup_{p\in I}p\mathbb{Z}$ a multiplicatively closed subset of $\mathbb{Z}$ and $M = S^{-1}\mathbb{Z}$. Then $M$ is a $\mathbb{Z}$-submodule of $\mathbb{Q}$, and hence $M$ is a cu-uniserial $\mathbb{Z}$-module.
}
\end{example}

By \cite[Proposition 2.4]{BehboodiVU}, every (virtually) uniserial $R$-module is uniform and B\'ezout. In the following theorem we show that this fact is hold for cu-uniserial modules.

\begin{thm} \label{222}
Let $M$ be a cu-uniserial left $R$-module. Then $M$ is uniform and B\'ezout.
\end{thm}
\begin{proof}
{Suppose to the contrary, $M$ is a cu-uniserial non-uniform left $R$-module and $Rx, Ry$ are nonzero cyclic submodules of $M$ such that $Rx \cap Ry = 0$. Then $(Rx \oplus Ry)/$Rad$(Rx \oplus Ry) \cong (Rx/$Rad$(Rx)) \oplus (Ry/$Rad$(Ry))$. Hence $(Rx\oplus Ry)/$Rad$(Rx\oplus Ry)$ is not uniform. This means that $M$ is not cu-uniserial, a contradiction. Thus every cu-uniserial left $R$-module $M$ is uniform. Next, we show that $M$ is B\'ezout. Let $N$ be a nonzero finitely generated submodule of $M$. Since $M$ is a cu-uniserial left $R$-module, $N/$Rad$(N)$ is cyclic. Therefore, there exists $x\in N$ such that $N/$Rad$(N) = \langle x + $Rad$(N) \rangle$. Hence, for every $n\in N$, $n + $Rad$(N) = rx + $Rad$(N)$, for some $r \in R$. Thus $n = (n - rx) + rx \in $Rad$(N) + Rx$, and so Rad$(N) + Rx = N$. This means that $N = Rx$, since Rad$(N)$ is superfluous in $N$, by \cite[Proposition 21.6 (4)]{Wisbauer}.}
\end{proof}

In the next result, we study left cu-uniserial rings $R$ which are  principal left ideal domains.
First, recall from \cite[Definition 7.5]{Lam} that an $R$-module $M$ is said to be \textit{singular} (resp., \textit{nonsingular}) if $Z(M) = M$ (resp., $Z(M) = 0$), where the \textit{singular submodule} $Z(M)$ of a left (resp., right) $R$-module $M$ consisting of
elements whose annihilators are essential left (resp., right) ideals in $R$. Also, if every left ideal of $R$ is projective, $R$ is called \textit{left hereditary}.

\begin{prop} \label{2.3}
Let $R$ be a left cu-uniserial ring. If $R$ is either left Noetherian and left nonsingular, or left hereditary, then $R$ is a  principal left ideal domain. Moreover, if $R$ is a principal left ideal domain with ${\rm J}(R) = 0$, then $R$ is a left cu-uniserial ring.
\end{prop}
\begin{proof}
{Let $R$ be a left cu-uniserial ring. First suppose that $R$ is a left Noetherian and left nonsingular ring.
It follows from Theorem \ref{222} that the left $R$-module $_{R}R$ is uniform and B\'ezout, and hence $R$ is a domain, by \cite[Corollary 3.25]{Goodearl}. Since $R$ is a left Noetherian and B\'ezout ring, it is a principal left ideal domain. Now, suppose that $R$ is a left hereditary ring. Then $R$ is a left nonsingular ring, by \cite[Examples 7.6 (8)]{Lam}. But $_RR$ is uniform, by Theorem \ref{222}, hence u.dim($_RR) = 1 < \infty$. Thus $R$ is a left Noetherian ring, by \cite[Corollary 7.58]{Lam}, and so $R$ is a  principal left ideal domain.

For the ``Moreover" statement, assume that $R$ is a principal left ideal domain with $J(R) = 0$. Then by \cite[Proposition 2.2]{BehboodiSeveral}, $_{R}R$ is virtually simple, and hence $R/$J$(R)$ is cyclic and uniform. Thus $R$ a left cu-uniserial ring.
}
\end{proof}

 The class of uniserial modules, virtually uniserial modules and cu-uniserial modules are not coincide. In the following results we give several classes of rings for which these classes of modules are coincide.

\begin{thm} \label{prev}
Suppose that $R$ is a finite direct product of local rings. Then every cu-uniserial left $R$-module is (virtually) uniserial. In particular, If $R$ is a local or an Artinian ring, then every cu-uniserial left $R$-module is (virtually) uniserial.

\end{thm}
\begin{proof}
{Suppose that $R$ is a finite direct product of local rings. It is sufficient to prove the result in the case $R = R_1 \times R_2$, in which $R_1$ and $R_2$ are local rings. Let $M$ be a cu-uniserial left $R$-module and $N$ a non-zero finitely generated submodule of $M$. Then $M$ is uniform and B\'ezout, by Theorem \ref{222}. This means that $N$ is cyclic, and hence there exist left ideals $I_1$ of $R_1$ and $I_2$ of $R_2$ such that $N \cong (R_1 \times R_2)/(I_1 \times I_2) \cong (R_1/I_1) \times (R_2/I_2)$. Since $N$ is uniform, either $N\cong R_1/I_1$ or $N\cong R_2/I_2$. Therefore, $N$ is a local submodule. Hence $N/$Rad$(N)$ is a simple submodule of $M$, and thus $M$ is (virtually) uniserial, by \cite[Proposition 55.1 $(d)\Rightarrow (a)$]{Wisbauer}.

The ``In particular" statement is obvious.
}
\end{proof}

\begin{thm} \label{3437}
Let $R$ be a ring. Then the following statements are equivalent:

$(a)$ $R$ is a left uniserial ring.

$(b)$ $R$ is a local left cu-uniserial ring.

$(c)$ $R$ is a semilocal left cu-uniserial ring.

$(d)$ $R$ is a local left virtually uniserial ring.
\end{thm}
\begin{proof}
{ $(a)\Rightarrow (b)\Rightarrow (c)$ are obvious.

$(a)\Leftrightarrow (d)$ It follows from \cite[Theorem 2.10]{BehboodiVU}.

$(c)\Rightarrow (a)$ Let $R$ be a semilocal left cu-uniserial ring. Then $R/$J$(R)$ is a cyclic uniform left $R$-module. It is easy to see that $_{R/ {\rm J}(R)}R/$J$(R)$ is a left cyclic uniform $R/$J$(R)$-module, since J$(R)(R/$J$(R)) = 0$ and $R/$J$(R)$ is a cyclic uniform left $R$-module. Since $R$ is a semilocal ring, $R/$J$(R)$ is a semisimple Artinian ring, and so $R/$J$(R)\cong M_{n_1}(D_1) \times \cdots \times M_{n_k}(D_k)$ for some positive integers $n_i$ and some division rings $D_i$, by \cite[Theorem 4.4 $(e)$]{GoodearlWarfield}. Therefore, $R/$J$(R)$ is a division ring, since it is uniform. Hence, \cite[Proposition 15.15 $(a)$]{AndersonFuller} implies that $R$ is a local ring. Thus $R$ is a left uniserial ring, by Theorem \ref{prev}.
}
\end{proof}
  In light of Theorems \ref{prev} and \ref{3437}, we state the following corollaries.

\begin{cor} \label{2.11}
Let $R$ be a left cu-uniserial ring. Then either $|{\rm Max}_l(R)| = \infty$ or $R$ is local.

\end{cor}
\begin{proof}
{ The proof follows from Theorem \ref{3437}.
}
\end{proof}

\begin{cor}
Let $(R, M )$ be a local ring. Then the following statements are equivalent:

$(a)$ Every proper ideal of $R$ is cu-serial.

$(b)$ There exist a uniserial module $U$ and a semisimple module $T$ such that $M = U \oplus T$.

$(c)$ There exist a uniserial module $U$ and a semisimple module $T$ such that each proper ideal of $R$ is isomorphic to $U' \oplus T'$ for some $U' \leq U$ and $T' \leq T$.

\end{cor}
\begin{proof}
{ The proof follows from Theorem \ref{3437} and \cite[Theorem 3.1]{Heidariserial}.
}
\end{proof}

Next, left cu-uniserial Noetherian rings are characterized.
\begin{thm} \label{vieww}
Let $R$ be a Noetherian ring. Then the following statements are equivalent:

$(a)$ $R$ is a left cu-uniserial ring.

$(b)$ $R$ is either a left uniserial ring of finite length or a principal left ideal domain with {\rm J}$(R) = 0$ or a local principal left ideal domain.

$(c)$ $R$ is either a left uniserial ring or a principal left ideal domain with ${\rm J}(R) = 0$.

$(d)$ $R$ is either a left uniserial ring or a principal ideal domain with ${\rm J}(R) = 0$.

\end{thm}
\begin{proof}
{ $(a)\Rightarrow (b)$ Let $R$ be a Noetherian left cu-uniserial ring. Then $_RR$ is B\'ezout, by Theorem \ref{222}. Hence $R$ is a principal left ideal ring. It follows from \cite[Theorem C]{Goldie} that $R$ is either a prime ring or a primary ring. If $R$ is a primary ring, then $R$ is an Artinian ring. Hence $R$ is left uniserial, by Theorem \ref{prev}. If $R$ is a prime ring, then $R$ is a principal left ideal domain, by \cite[Exercise 7M]{GoodearlWarfield}. If J$(R) = 0$, then we are done. So, assume that J$(R) \neq 0$. Since $R$ is a Noetherian ring, $R$ is a principal ideal domain, by \cite[ Corollary 3.7]{Robson}. But every principal ideal domain is a Dedekind domain, hence $R/$J$(R)$ is an Artinian principal ideal ring, by \cite[Theorem 5.7.10]{McConnell}. It is easy to see that $_{R/ {\rm J}(R)}R/$J$(R)$ is a left cyclic uniform $R/$J$(R)$-module, since J$(R)(R/$J$(R)) = 0$ and $R/$J$(R)$ is a cyclic uniform left $R$-module. Hence $R/$J$(R)$ is a local ring, by \cite[Corollary 2.2]{Eisenbud}, which means that $(R/$J$(R))/$J$(R/$J$(R)) = R/$J$(R)$ is a semisimple Artinian ring. Thus $R/$J$(R)\cong M_{n_1}(D_1) \times \cdots \times M_{n_k}(D_k)$ for some positive integers $n_i$ and some division rings $D_i$, by \cite[Theorem 4.4 $(e)$]{GoodearlWarfield}. Therefore, $R/$J$(R)$ is a division ring, since it is uniform. Hence $R$ is a local ring, by \cite[Proposition 15.15]{AndersonFuller}.

$(b)\Rightarrow (c)$ It follows from \cite[Lemma 2.13]{BehboodiVU}.

$(c)\Rightarrow (d)$ It follows from \cite[Corollary 3.7]{Robson}.

$(d)\Rightarrow (a)$ It follows from Proposition \ref{2.3}.}
\end{proof}

Since the uniserial property for a ring $R$ is not left/right symmetric, it follows from Theorem \ref{3437} that the cu-uniserial property for a ring $R$ is not left/right symmetric. In the following result we give a class of rings for which the cu-uniserial ring property is left/right symmetric.

\begin{cor} \label{29}
Let $R$ be a nonlocal Noetherian ring. Then the following statements are equivalent:

$(a)$ $R$ is a left cu-uniserial ring.

$(b)$ $R$ is a cu-uniserial ring.

$(c)$ $R$ is a principal ideal domain with {\rm J}$(R) = 0$.
\end{cor}
\begin{proof}
{ $(a)\Rightarrow (c)$ Since $R$ is a nonlocal Noetherian ring, it follows from Theorem \ref{3437} that $R$ is not a left uniserial ring, and so Theorem \ref{vieww} implies that $R$ is a principal ideal domain with J$(R) = 0$.

$(c)\Rightarrow (a)$ It follows from Proposition \ref{2.3}.

$(b)\Rightarrow (a)$ It is clear.

$(c)\Rightarrow (b)$ Assume that $R$ is a principal ideal domain with $J(R) = 0$. Then by \cite[Proposition 2.2]{BehboodiSeveral}, $R_{R}$ is virtually simple, and so $R/$J$(R)$ is cyclic and uniform. Hence $R$ is a right cu-uniserial ring, and thus $R$ is a cu-uniserial ring, by Proposition \ref{2.3}.}
\end{proof}

Recall that a ring $R$ is called left semi-duo if every maximal left ideal of $R$ is two-sided.

\begin{prop} \label{2.112}
Suppose that $R$ is a Noetherian ring. If $R$ is a left cu-uniserial ring, then $R$ is either a left uniserial ring or a principal ideal domain with $|{\rm Max}_l(R)|= \infty$. The converse is true if $R$ is a left semi-duo ring.
\end{prop}

\begin{proof}
{Suppose that $R$ is a Noetherian left cu-uniserial ring. Then $R$ is either a left uniserial ring or a principal left ideal domain with $J(R) = 0$, by Theorem \ref{vieww}. Assume that $R$ is a principal left ideal domain that is not left uniserial, $J(R) = 0$ and it has finitely many maximal left ideals. This means that $R$ is a semilocal ring. Since $R$ is a left cu-uniserial ring, Theorem \ref{3437} implies that $R$ is a left uniserial ring, a contradiction. Conversely, assume that $R$ is a principal ideal domain with $|$Max$_l(R)| = \infty$ such that $J(R) \neq 0$. Then $R/$J$(R)$ is an Artinian ring, by \cite[Theorem 5.7.10]{McConnell}. Since $R$ is left semi-duo, the maximal left ideals of $R/$J$(R)$ are precisely the minimal prime ideals of $R/$J$(R)$, by \cite[Proposition 4.19]{GoodearlWarfield}, whereas $R/$J$(R)$ has only finitely many minimal prime ideals, by \cite[Theorem 3.4]{GoodearlWarfield}, and hence $|$Max$_l(R)| < \infty$, a contradiction. Thus $R$ is a left cu-uniserial ring, by Corollary \ref{29}.}
\end{proof}

\begin{prop} \label{2.17}
Suppose that $R$ is a left cu-serial ring such that every proper ideal of $R$ is left cu-serial. Then $R$ is either a left uniserial ring or a principal left ideal domain with $|{\rm Max}_l(R)|= \infty$.
\end{prop}

\begin{proof}
{Let $R$ be a left cu-uniserial ring such that every proper ideal of $R$ is left cu-serial. Then $R$ is either local or $|{\rm Max}_l(R)| = \infty$, by Corollary \ref{2.11}. If $R$ is local, then $R$ is a left uniserial ring, by Theorem \ref{3437}. So, suppose that $|{\rm Max}_l(R)| = \infty$. Since every proper ideal $I$ of $R$ is left cu-serial, every proper ideal $I$ of $R$ is direct sum of left cu-uniserial modules, and hence every proper ideal $I$ of $R$ is direct sum of cyclic uniform modules. Thus $R$ is a principal ideal ring, by \cite[Theorem 3.4]{Asgaricu}.
}
\end{proof}

Recall that if every left ideal of a ring $R$ is two-sided, then $R$ is said to be left duo ring.

\begin{thm} \label{2.115}
Let $R$ be a Noetherian left duo ring and $M$ a left $R$-module. Then $M$ is finitely generated cu-uniserial if and only if $M$ is either a uniserial left $R$-module with finite length, or a local virtually simple left $R$-module or a virtually simple left $R$-module with $|{\rm Max}_l(M)| = \infty$.
\end{thm}
\begin{proof}
{Let $R$ be a Noetherian left duo ring and $M$ a finitely generated cu-uniserial left $R$-module. Then $M$ is a cyclic uniform Noetherian B\'ezout module, by Theorem \ref{222}. This means that $M\cong R/I$ such that $I$ is an ideal of $R$, and so $R/I$ is a uniform left $R/I$-module and $R/I$ is a principal left ideal ring. Hence the ring $R/I$ is either primary or prime, by \cite[Theorem C]{Goldie}. If $R/I$ is primary, then $R/I$ is an Artinian left cu-uniserial ring. Hence $R/I$ is a finite direct product of local rings, and so $R/I$ is a uniserial left $R/I$-module with finite length, by Theorem \ref{prev}. If $R/I$ is a prime principal left ideal ring, then $R/I$ is principal left ideal domain, by \cite[Exercise 7M]{GoodearlWarfield}. Therefore, \cite[Proposition 2.2]{BehboodiSeveral} implies that $R/I$ is a virtually simple module, and thus $R/I$ is a local virtually simple left $R$-module or a virtually simple left $R$-module with $|{\rm Max}_l(R/I)| = \infty$, by Corollary \ref{2.11}. Conversely, suppose that $M$ is either a local virtually simple left $R$-module or a virtually simple left $R$-module with $|{\rm Max}_l(R)| = \infty$. Then $M$ is cyclic, and so $M\cong R/I$ such that $I$ is an ideal of $R$ and $R/I$ is a principal left ideal domain which is either local or with $|{\rm Max}_l(R/I)| = \infty$. Thus, Proposition \ref{2.112} implies that $R/I$ is a cu-uniserial left $R$-module.
}
\end{proof}

In the following proposition we see that if a non uniserial cu-uniserial left $R$-module is virtually simple, its radical is zero.

\begin{prop} \label{23}
Suppose that $R$ is a left duo ring and $M$ is a non uniserial cu-uniserial virtually simple left $R$-module. Then {\rm Rad}$(M) = 0$.
\end{prop}
\begin{proof}
{Suppose that $R$ is a left duo ring and $M$ is a non uniserial cu-uniserial virtually simple left $R$-module. Then $M$ is cyclic, and so there exists a left ideal $I$ of $R$ such that $M \cong R/I$. Hence $R/I$ is a principal left ideal domain, by \cite[Proposition 2.2]{BehboodiSeveral}. Since $M$ is a cu-uniserial left $R$-module, it follows from Corollary \ref{2.11} that either $|{\rm Max}_l(R/I)| = \infty$ or $R/I$ is local. Since \cite[Lemma 2.13]{BehboodiVU} implies that every local principal left ideal ring is a left uniserial ring, we can assume that $|{\rm Max}_l(R/I)| = \infty$. If J$(R/I) \neq 0$, then $(R/I)/$J$(R/I)$ is an Artinian ring, by \cite[Theorem 5.7.10]{McConnell}. Since $R$ is left duo, the maximal left ideals of $(R/I)/$J$(R/I)$ are precisely the minimal prime ideals of $(R/I)/$J$(R/I)$, by \cite[Proposition 4.19]{GoodearlWarfield}, whereas $(R/I)/$J$(R/I)$ has only finitely many minimal prime ideals, by \cite[Theorem 3.4]{GoodearlWarfield}, and hence $|{\rm Max}_l(R/I)| < \infty$, a contradiction. Thus J$(R/I) = 0$.
}
\end{proof}

We close this section with the following corollary which shows that in a left duo ring $R$ the class of virtually simple virtually uniserial left $R$-modules and  virtually simple cu-uniserial left $R$-modules are coincide.

\begin{cor}
Let $R$ be a left duo ring and $M$ a virtually simple left $R$-module.Then the following statements are equivalent:

$(a)$ $M$ is cu-uniserial.

$(b)$ $M$ is virtually uniserial.

$(c)$ $M$ is either local or $|{\rm Max}_l(R/I)| = \infty$.

$(d)$ $M$ is uniserial or {\rm Rad}$(M) = 0$.
\end{cor}
\begin{proof}
{The proof follows from Theorem \ref{2.115}, Proposition \ref{23} and \cite[Corollary 2.25]{BehboodiVU}.
}
\end{proof}
\begin{center}{\section{Finitely generated cyclic-uniform serial modules
}}\end{center}

The purpose of this section is to characterize rings each of whose finitely generated module is cu-serial.
First, we provide a characterization of a ring $R$ has every $R$-module is cu-serial.

\begin{thm} \label{311}
Let $R$ be a ring whose every left $R$-module is cu-serial. Then $R$ is Artinian. The converse is true if $R$ is serial.
\end{thm}
\begin{proof}
{Let $R$ be a ring whose every left $R$-module is cu-serial. First, we show that $R$ is left Artinian. It follows from Theorem \ref{222} that every left $R$-module is a direct sum of cyclic uniform modules. Hence $R$ is a left Artinian ring, by \cite[Theorem 3.4D]{Faith}. Now, we show that $R$ is also right Artinian. Suppose that $M$ is an indecomposable finitely generated left $R$-module. Then $M$ is cyclic uniform, by Theorem \ref{222}. Therefore,  $M \cong R/I$ for some left ideal $I$ of $R$. This means that length$(M)+$ length$(I)$ = length$(_RR) < \infty$, and so length$(M) \leq$ length$(_RR) < \infty$. Therefore, the left Artinian ring $R$ is of bounded representation type, and hence $R$ is left and right pure-semisimple, by \cite[Proposition 54.3]{Wisbauer}. Thus $R$ is a right Artinian, by \cite[Proposition 53.6 (2)(i)]{Wisbauer}. Conversely, suppose that $R$ is an Artinian serial ring. It follows from Nakayama-Skornyakov's theorem (see \cite[Theorem 17]{Nakayama} and \cite{Skornyakov}) that every left $R$-module is serial, and thus every left $R$-module is cu-serial.
}
\end{proof}

Now, we have the following immediate corollaries.
\begin{cor} \label{322}
Let $R$ be a ring. Then the following statements are equivalent:

$(a)$ Every left $R$-module is cu-serial.

$(b)$ Every left $R$-module is serial.

$(c)$ Every left $R$-module is virtually serial.

$(d)$ $R$ is an Artinian serial ring.
\end{cor}
\begin{proof}
{The proof follows from Theorems \ref{prev}, \ref{311} and \cite[Theorem 3.2]{BehboodiVU}.
}
\end{proof}

\begin{cor} \label{333}
Let $R$ be a ring such that all idempotents are central. Then the following statements are equivalent:

$(a)$ Every $R$-module is cu-serial.

$(b)$ Every $R$-module is serial.

$(c)$ $R$ is a K\"othe ring.

$(d)$ $R$ is isomorphic to a finite product of Artinian uniserial rings.
\end{cor}
\begin{proof}
{The proof follows from Corollary \ref{322} and \cite[Corollary 3.3]{BehboodiKothe}.
}
\end{proof}

\begin{thm} \label{2.113}
Suppose that $R$ is a ring such that every left $R$-module is left cu-serial. Then $R/{\rm J}(R)\cong M_{n_1}(D_1) \times \cdots \times M_{n_k}(D_k)$ for some positive integers $n_i$ and some principal left ideal domains $D_i$.
\end{thm}

\begin{proof}
{Suppose that $R$ is a ring such that every left $R$-module is left cu-serial. Then $R$ is a left Artinian ring, by Theorem \ref{311}. Hence $R/$J$(R)$ is semisimple, by \cite[Proposition 15.16]{AndersonFuller}. This means that $R/$J$(R)$ is a direct sum of simple $R$-modules, and thus $R/$J$(R)\cong M_{n_1}(D_1) \times \cdots \times M_{n_k}(D_k)$ for some positive integers $n_i$ and some principal left ideal domains $D_i$, by \cite[Theorem 3.11]{Behboodicu}.
}
\end{proof}

Recall that a ring $R$ is said to be left V-ring if each simple left $R$-module is injective. Also, a module $M$ is said to be extending if every closed submodule of $M$ is a direct summand of $M$.

\begin{thm} \label{2.114}
Suppose that $R$ is a V-ring such that every left $R$-module is free and finitely generated. Then the following statements are equivalent:

$(a)$ Every left $R$-module is left cu-serial with {\rm J}$(R) = 0$.

$(b)$ $R\cong M_{n_1}(D_1) \times \cdots \times M_{n_k}(D_k)$ for some positive integers $n_i$ and some principal left ideal domains $D_i$ with {\rm J}$(D_i) = 0$.

\end{thm}
\begin{proof}
{ $(a)\Rightarrow (b)$ Assume that $R$ is a ring with {\rm J}$(R) = 0$ such that every left $R$-module is left cu-serial, free and finitely generated. It follows from Proposition \ref{2.114} that $R\cong M_{n_1}(D_1) \times \cdots \times M_{n_k}(D_k)$ for some positive integers $n_i$ and some principal left ideal domains $D_i$. Since J$(R) = $ J$(M_{n_1}(D_1) \times \cdots \times M_{n_k}(D_k)) = 0$, J$(D_i) = 0$, by \cite[Corollary 17.13]{AndersonFuller}.

$(b)\Rightarrow (a)$ Suppose that $R\cong M_{n_1}(D_1) \times \cdots \times M_{n_k}(D_k)$ for some positive integers $n_i$ and some principal left ideal domains $D_i$ with {\rm J}$(D_i) = 0$. It follows from \cite[Theorem 2.19]{BehboodiVU} that $R$ is a left virtually serial ring with J$(R) = 0$, and so $R = \bigoplus_{i=1}^nI_i$, where $I_i$'s are left virtually uniserial $R$-modules. Therefore, by  \cite[Proposition 21.6]{Wisbauer}, $R = R/$J$(R) = \bigoplus_{i=1}^nI_i/$Rad$(I_i)$ is a direct sum of virtually simple $R$-modules, and so $R$ is a left completely virtually semisimple ring, by \cite[Theorem 3.11]{Behboodicu}. Since every left $R$-module is free, we conclude by \cite[Proposition 3.3]{Behboodicu} that every left $R$-module is virtually semisimple. This implies that every left $R$-module $M$ is a direct sum of virtually simple modules, by \cite[Corollary 2.3]{Behbooditwo}, whence $M = \bigoplus_{i=1}^nM_i$ for some virtually simple modules $M_i$. Thus, by \cite[Theorem 3.75]{Lam}, $M = M/$Rad$(M) = \bigoplus_{i=1}^nM_i/$Rad$(M_i)$ is a direct sum of virtually simple $R$-modules, because $R$ is a V-ring. Since each $M_i$ is a finitely generated left $R$-module and Rad$(M_i) = 0$, it follows from \cite[Proposition 2.6]{BehboodiVU} that each $M_i$ is virtually uniserial. Hence $M$ is a  direct sum of virtually uniserial modules. Therefore, $M$ is a direct sum of cu-uniserial modules, and thus every left $R$-module is left cu-serial.
}
\end{proof}

\begin{cor}
 Suppose that $R$ is a V-ring such that every left $R$-module is left cu-serial, free and finitely generated. Then every left $R$-module is extending.
\end{cor}
\begin{proof}
{ Suppose that $R$ is a V-ring such that every left $R$-module is left cu-serial, free and finitely generated. Then $R\cong M_{n_1}(D_1) \times \cdots \times M_{n_k}(D_k)$ for some positive integers $n_i$ and some principal left ideal domains $D_i$ with {\rm J}$(D_i) = 0$, by {\rm Theorem \ref{2.114}}. Hence $R$ is a left completely virtually semisimple ring, by {\rm \cite[Theorem 3.11]{Behboodicu}}. Thus, by {\rm \cite[Corollary 3.5]{BehboodiSeveral}}, every left $R$-module is extending.
}
\end{proof}

In the following result, we show that if every finitely generated left $R$-module is cu-serial, the ring $R$ is a left FGC ring of finite uniform dimension.

\begin{thm} \label{344}
Let $R$ be a ring whose every finitely generated left $R$-module is cu-serial. Then $R$ is a left FGC ring of finite uniform dimension. The converse is not true in general.
\end{thm}
\begin{proof}
{Let $R$ be a ring whose every finitely generated left $R$-module is cu-serial and $M$ a finitely generated left $R$-module. Then $M = \bigoplus_{i\in I}N_i$, where $N_i$ are cu-uniserial modules and $I$ is a finite index set. Since $M$ is finitely generated, every $N_i$ is also finitely generated, and hence all $N_i$'s are cyclic, since they are B\'ezout. Thus $M$ is a direct sum of cyclic modules, and hence $R$ is a left FGC ring. In the other hand, $_RR = \bigoplus_{i\in I}N_i$ for some cu-uniserial modules $N_i$ and  a finite index set $I$. Thus $_RR$ is a direct sum of uniform modules, and hence $_RR$ has finite uniform dimension, by \cite[Proposition 5.15]{GoodearlWarfield}. However, the converse statement is not true, since, for example, every non local semilocale commutative principal ideal domain $R$ is an FGC ring with u.dim$(R)  < \infty$, but $R$ is not cu-uniserial.
}
\end{proof}
\begin{prop} \label{355}
Let $R$ be a commutative ring. Then $R$ is a nonlocal cu-uniserial FGC domain such that every proper ideal of $R$ is cu-serial if and only if $R$ is a principal ideal domain with $|{\rm Max}(R)| = \infty$.
\end{prop}
\begin{proof}
{Let $R$ be a commutative nonlocal cu-uniserial FGC domain such that every ideal of $R$ is cu-serial. Then $R$ is either a uniserial ring or a principal ideal domain with $|{\rm Max}(R)|= \infty$, by Proposition \ref{2.17}. Since $R$ is nonlocal, it follows from Theorem \ref{3437} that $R$ is not a uniserial ring, and so, the proof is complete. Conversely, suppose that $R$ is a principal ideal domain with $|{\rm Max}(R)| = \infty$. Then $R$ is a cu-uniserial ring, by Proposition \ref{2.112}. Let $I$ be a proper ideal of $R$. Then for every finitely generated ideal $J$ of $I$, $J$ is also a finitely generated ideal of $R$. Hence $J/$Rad$(J)$ is cyclic and uniform, because $R$ is a cu-uniserial ring. Thus every proper ideal of $R$ is cu-serial, and $R$ is an FGC domain, since every principal ideal domain is an FGC ring.
}
\end{proof}
Proposition \ref{355} leads to the next corollary.
\begin{cor} \label{366}
Let $R$ be a commutative ring with $|{\rm Max}(R)| = \infty$. Then the following statements are equivalent:

$(a)$ $R$ is a nonlocal cu-uniserial FGC domain such that every proper ideal of $R$ is cu-serial.

$(b)$ $R$ is a principal ideal domain.

$(c)$ Every proper ideal of $R$ is a direct sum of cyclic uniform modules.

$(d)$ Every proper ideal of $R$ is a direct sum of completely cyclic modules.
\end{cor}
\begin{proof}
{The proof follows from Proposition \ref{355} and \cite[Theorem 3.4]{Asgaricu}.
}
\end{proof}

\begin{thm} \label{jdid}
Let $R$ be a commutative Noetehrian V-ring. If every proper ideal of $R$ is cu-serial, then $R$ is a principal ideal ring, or a local ring with the unique maximal ideal $M$ such that $M = Rx \oplus S$, where $Rx$ is a completely cyclic module and $S$ is a semisimple module.

\end{thm}
\begin{proof}
{Let $I$ be a proper ideal of $R$. Then $I = \bigoplus_{i=1}^nI_i$ is a direct sum of cu-uniserial $R$-modules. Since $R$ is a Noetehrian and V-ring, it follows from Theorem \ref{222} that every submodule $J_i = J_i/$Rad$(J_i)$ of $I_i$ is cyclic uniform for each $1\leq i\leq n$, and so $I$ is a direct sum of completely cyclic modules. Hence, $R$ is a principal ideal ring, or a local ring with the unique maximal ideal $M$ such that $M = Rx \oplus S$, where $Rx$ is a completely cyclic module and $S$ is a semisimple module, by \cite[Theorem 3.4]{Asgaricu}.
}
\end{proof}
In the following result we give a characterization of rings each of whose finitely generated module is cu-serial.

\begin{thm} \label{377}
Let $R$ be a commutative ring. Then every finitely generated $R$-module is cu-serial if and only if $R = \prod_{i=1}^nR_i$ where $n > 0$ is a positive integer, $R_i$'s are almost maximal uniserial rings, principal ideal domains with {\rm J}$(R_i) = 0$ and torch rings such that $R_i/P_i$'s are principal ideal domains with $|{\rm Max}(R_i/P_i)| = \infty$, where every $P_i$ is the unique minimal prime ideal of $R_i$.
\end{thm}
\begin{proof}
{Let $R$ be a commutative ring such that every finitely generated $R$-module is cu-serial. Then $R$ is an FGC ring, by Theorem \ref{344}. Hence, \cite[Theorem 9.1]{Brandal} implies that $R$ is a finite direct product of indecomposable rings. Therefore, without loss of generality, assume that $R$ is an indecomposable FGC ring, which means that $R$ is a cu-uniserial ring. This implies that either $|{\rm Max}(R)| = \infty$ or $R$ is local, by Corollary \ref{2.11}. If $R$ is local, we conclude by \cite[Theorem 4.5]{Brandal} and Theorem \ref{3437} that $R$ is an almost maximal uniserial ring. If $|{\rm Max}(R)| = \infty$, it follows from \cite[Theorem 9.1]{Brandal} that $R$ is either an almost maximal B\'ezout domain or a torch ring. If $R$ is an almost maximal B\'ezout domain, then $R$ is a principal ideal domain with J$(R)= 0$, by Propositions \ref{355} and \ref{2.112}. If $R$ is torch with a unique minimal prime ideal $P$, then $R$ is an FGC ring, by \cite[Theorem 9.1]{Brandal}, and so $R/P$ is an FGC cu-uniserial domain, by \cite[Theorem 4.1(2)]{Brandal}. Again, by \cite[Lemma 5.3(6)]{Brandal}, $P^2 = 0$, and hence $P \subseteq$ J$(R)$, which means that $R/P$ is nonlocal. Thus $R$ is a torch ring such that $R/P$ is a principal ideal domain with $|{\rm Max}(R/P)| = \infty$, where $P$ is the unique minimal prime ideal of $R$, by Propositions \ref{355}. Conversely, since every virtually serial $R$-module is cu-serial, the proof follows from \cite[Theorem 3.8]{BehboodiVU}.
}
\end{proof}

\begin{thm} \label{388}
Let $R$ be a ring. Then the following statements are equivalent:

$(a)$ $R$ is a left V-ring such that every finitely generated left $R$-module is cu-serial and every cyclic $R$-module is either virtually simple direct summand of $_RR$ or simple.

$(b)$ $R\cong M_{n_1}(D_1) \times \cdots \times M_{n_k}(D_k)$ for some positive integers $n_i$ and some principal ideal V-domains $D_i$.

$(c)$ $R$  is a fully virtually semisimple ring.

$(d)$ Every finitely generated left $R$-module is a direct sum of virtually simple $R$-modules.

$(e)$ Every finitely generated left $R$-module is virtually semisimple.

$(f)$ Every finitely generated left $R$-module is completely virtually semisimple.
\end{thm}
\begin{proof}
{ $(a)\Rightarrow (b)$ Let $M$ be a finitely generated left $R$-module. Then $_RM =$ $_RN_1 \oplus \cdots \oplus$ $_RN_n$ for some finitely generated cu-uniserial left $R$-modules $_RN_i$. Hence $M$ is a direct sum of cyclic uniform left $R$-modules, since $R$ is a left V-ring and Rad$(N_i) = 0$ for each $1 \leq i \leq n$, by \cite[Theorem 3.75]{Lam}. Thus, by hypothesis, $M$ is a direct sum of cyclic $R$-modules that are either virtually simple direct summand of $_RR$ or simple left $R$-modules. Therefore, $R\cong M_{n_1}(D_1) \times \cdots \times M_{n_k}(D_k)$ for some positive integers $n_i$ and some principal ideal V-domains $D_i$, by \cite[Theorem 3.4]{BehboodiSeveral}.

$(b)\Rightarrow (a)$ Suppose that $R\cong M_{n_1}(D_1) \times \cdots \times M_{n_k}(D_k)$ for some positive integers $n_i$ and some principal ideal V-domains $D_i$. Since every $D_i$ is a V-ring, it follows from \cite[Theorem 17.20]{Lam} and \cite[Propositions 21.6 and 21.8]{AndersonFuller} that every $M_{n_i}(D_i)$ is also a V-ring, and hence \cite[Proposition 23.4]{Wisbauer} implies that $R$ is a left $V$-ring. Let $M$ be a finitely generated left $R$-module. Then $M$ is a direct sum of virtually simple $R$-modules $N$, by \cite[Theorem 3.4]{BehboodiSeveral}. But Rad$(N) = 0$ for each $R$-module $N$, by \cite[Theorem 3.75]{Lam}. This means that $K/$Rad$(K) = K\cong N$ for every finitely generated submodule $K$ of $N$, and so $N$ is virtually uniserial. Thus $M$ is a direct sum of cu-uniserial $R$-modules, and hence $M$ is cu-serial. In the other hand, every cyclic $R$-module is either virtually simple direct summand of $_RR$ or simple, by \cite[Theorem 3.4(3)]{BehboodiSeveral}.

$(b)\Leftrightarrow (d)$ It follows from \cite[Theorem 3.4(3)]{BehboodiSeveral}.

$(b)\Leftrightarrow (c)$ It follows from \cite[Corollary 3.8]{BehboodiSeveral}.

$(b)\Leftrightarrow (e) \Leftrightarrow (f)$ Theses follow from \cite[Theorem 2.8]{BehboodiSeveral}.
}
\end{proof}

\begin{cor} \label{399}
Let $R$ be a left V-ring such that every finitely generated left $R$-module is cu-serial and every cyclic $R$-module is either virtually simple direct summand of $_RR$ or simple. Then every finitely generated left $R$-module is a direct sum of a projective module and a singular (injective) semisimple module.
\end{cor}
\begin{proof}
{The proof follows from Theorem \ref{388} and \cite[Proposition 2.6]{BehboodiSeveral}.
}
\end{proof}

\begin{thm} \label{31010}
Let $R$ be a V-ring such that every left $R$-module is free and finitely generated. Then the following statements are equivalent:

$(a)$ Every left $R$-module is left cu-serial with {\rm J}$(R) = 0$.

$(b)$ Every finitely generated left $R$-module has a decomposition $M = Z(M)\oplus P$, where $P$ is a virtually serial projective $R$-module with {\rm Rad}$(P) = 0$.

\end{thm}
\begin{proof}
{ $(a)\Rightarrow (b)$ It follows from Theorem \ref{2.114} that $R\cong M_{n_1}(D_1) \times \cdots \times M_{n_k}(D_k)$ for some positive integers $n_i$ and some principal left ideal domains $D_i$ with {\rm J}$(D_i) = 0$. Hence, by \cite[Theorem 3.11]{Behboodicu}, $R$ is a left completely virtually semisimple ring. So, $M = Z(M)\oplus P$
 where $P$ is a direct sum of projective virtually simple $R$-modules, by \cite[Corollary 3.13]{Behboodicu}. Thus every direct summand of $P$ is finitely generated with radical equal to zero, by \cite[Proposition 17.10]{AndersonFuller}, and hence Rad$(P) = 0$. By a similar argument as in the proof of Theorem \ref{388} $(b)\Rightarrow (a)$, we can see that $P$ is a direct sum of cu-uniserial $R$-modules, and hence $M$ is cu-serial.

$(b)\Rightarrow (a)$ It follows from \cite[Proposition 3.11]{BehboodiVU} that $R$ is a virtually serial ring with J$(R) = 0$, and so $R = \bigoplus_{i=1}^nI_i$, where $I_i$'s are left virtually uniserial $R$-modules. Therefore, by \cite[Proposition 21.6]{Wisbauer}, $R = R/$J$(R) = \bigoplus_{i=1}^nI_i/$Rad$(I_i)$ is a direct sum of Virtually simple $R$-modules, and so $R$ is a left completely virtually semisimple ring, by \cite[Theorem 3.11]{Behboodicu}. Since every left $R$-module is free, we conclude by \cite[Theorem 3.3]{Behboodicu} that every left $R$-module is virtually semisimple. This implies that every left $R$-module $M$ is a direct sum of virtually simple modules, by \cite[Corollary 2.3]{Behbooditwo}, whence $M = \bigoplus_{i=1}^nM_i$ for some virtually simple modules $M_i$. Thus $M = M/$Rad$(M) = \bigoplus_{i=1}^nM_i/$Rad$M_i$ is a direct sum of Virtually simple $R$-modules, because $R$ is a V-ring. Since each $M_i$ is a finitely generated left $R$-module and Rad$(M_i) = 0$, it follows from \cite[Proposition 2.6]{BehboodiVU} that each $M_i$ is virtually uniserial. Hence $M$ is a  direct sum of virtually uniserial modules. Therefore, $M$ is a direct sum of cu-uniserial modules, and thus every left $R$-module is left cu-serial.
}
\end{proof}

In the following result a relationship between principal ideal rings $R$ and the direct summand of their left $R$-modules is established.

\begin{thm} \label{31111}
Let $R$ be a duo ring. Then $R$ is a principal ideal ring if and only if every finitely generated left $R$-module $M$ has a decomposition $M = Z\oplus P$ where $Z$ is Noetherian cu-serial and $P$ is a direct sum of projective virtually simple modules.

\end{thm}
\begin{proof}
{Since every serial left $R$-module is cu-serial, it suffices to show the converse (see \cite[Theorem 3.12]{BehboodiVU}). So, suppose that every finitely generated left $R$-module $M$ has a decomposition $M = Z\oplus P$ where $Z$ is Noetherian cu-serial and $P$ is a direct sum of projective virtually simple modules. If $M = R$, then $R = Z\oplus P$ where $Z$ is Noetherian cu-serial and $P$ is a direct sum of projective virtually simple modules, and this means that $R$ is Noetherian. Since every cu-uniserial left $R$-module is uniform and B\'ezout, by Theorem \ref{222}, every finitely generated left $R$-module is a direct sum of cyclic modules. Hence $R$ is a left FGC ring. Thus $R$ is a principal ideal ring, by \cite[Theorem 3.5]{GhorbaniNaji}.
}
\end{proof}
\begin{thm} \label{31112}
Let $R$ be a Noetherian duo ring. Then the following statements are equivalent:

$(a)$ $R$ is a cu-serial ring.

$(b)$ $R = \prod_{i=1}^nR_i$ such that $n \geq 0$ is a nonnegative integer and every $R_i$ is either a uniserial ring or a principal ideal domain with J$(R_i) = 0$.

$(c)$ Every finitely generated left and right $R$-module is cu-serial.

$(d)$ Every finitely generated left and right $R$-module is a direct sum of local virtually simple modules, virtually simple modules with infinitely many maximal submodules, and finite length uniserial modules.

$(e)$ Every finitely generated left and right $R$-module is a direct sum of virtually simple modules with zero radical and uniserial modules.
\end{thm}
\begin{proof}
{$(a)\Rightarrow (b)$ Let $R$ be a Noetherian cu-serial duo ring. Then $R$ is a principal ideal domain, by Proposition \ref{2.3}. Hence  $R = \prod_{i=1}^nR_i$ such that $n \geq 0$ and every $R_i$ is either a uniserial ring or a principal ideal domain with J$(R_i) = 0$, by Theorem \ref{vieww}.

$(c)\Rightarrow (d)$ Suppose that every finitely generated left and right $R$-module is cu-serial and $M$ is a finitely generated left (right) $R$-module. Then, by part $(c)$, $M$ is a direct sum of finitely generated cu-uniserial left (right)
$R$-module $N_i$. Hence, every $N_i$ is a direct sum of either  a uniserial left $R$-module with finite length, or a local virtually simple left $R$-module or a virtually simple left $R$-module with $|{\rm Max}(N_i)| = \infty$, by Theorem \ref{2.115}.

$(b)\Rightarrow (c)$, $(d)\Leftrightarrow (e)\Rightarrow (a)$ follows from \cite[Theorem 3.13]{BehboodiVU}.}
\end{proof}

We end this paper with the following corollary.

\begin{cor} \label{31113}
Let $R$ be a hereditary duo ring. Then the following statements are equivalent:

$(a)$ $R$ is a cu-serial ring.

$(b)$ $R = \prod_{i=1}^nR_i$ such that $n \geq 0$ and every $R_i$ is either a uniserial ring or a principal ideal domain with J$(R_i) = 0$.

$(c)$ Every finitely generated left and right $R$-module is cu-serial.

$(d)$ Every finitely generated left and right $R$-module is a direct sum of local virtually simple modules, virtually simple modules with infinitely many maximal submodules, and finite length uniserial modules.

$(e)$ Every finitely generated left and right $R$-module is a direct sum of virtually simple modules with zero radical and uniserial modules.
\end{cor}
\begin{proof}
{Let $R$ be a hereditary cu-serial duo ring. Then u.dim$(_RR) < \infty$, by Theorem \ref{222}. Hence $R$ is Noetherian, by \cite[Corollary 7.58]{Lam}, and thus the proof follows from Theorem \ref{31112}.}
\end{proof}

\end{document}